\documentclass[11pt]{amsart}

\newtheorem{theorem}{Theorem}[section]

\theoremstyle{definition}

\theoremstyle{remark}

\numberwithin{equation}{section}

\begin{document}

\newcommand{\spacing}[1]{\renewcommand{\baselinestretch}{#1}\large\normalsize}
\spacing{1.14}

\title{On the curvature of invariant Kropina metrics}

\author {H. R. Salimi Moghaddam}

\address{Department of Mathematics, Faculty of  Sciences, University of Isfahan, Isfahan,81746-73441-Iran.} \email{salimi.moghaddam@gmail.com and hr.salimi@sci.ui.ac.ir}



\begin{abstract}
In the present article we compute the flag curvature of a special
type of invariant Kropina metrics on homogeneous spaces.
\end{abstract}

\maketitle


\section{\textbf{Introduction}}\label{intro}
Let $M$ be a smooth $n-$dimensional manifold and $TM$ be its
tangent bundle. A Finsler metric on $M$ is a non-negative function
$F:TM\longrightarrow \Bbb{R}$ which has the following properties:
\begin{enumerate}
    \item $F$ is smooth on the slit tangent bundle
    $TM^0:=TM\setminus\{0\}$,
    \item $F(x,\lambda y)=\lambda F(x,y)$ for any $x\in M$, $y\in T_xM$ and $\lambda
    >0$,
    \item the $n\times n$ Hessian matrix $[g_{ij}(x,y)]=[\frac{1}{2}\frac{\partial^2 F^2}{\partial y^i\partial
    y^j}]$ is positive definite at every point $(x,y)\in TM^0$.
\end{enumerate}
For a smooth manifold $M$ suppose that $g$ and $b$ are a
Riemannian metric and  a 1-form respectively as follows:
\begin{eqnarray}
  g&=&g_{ij}dx^i\otimes dx^j \\
  b&=&b_idx^i.
\end{eqnarray}
An important family of Finsler metrics is the family of
$(\alpha,\beta)-$metrics which is introduced by M. Matsumoto (see
\cite{Ma}) and has been studied by many authors. An interesting
and important example of such metrics is the Kropina metrics with
the following form:
\begin{eqnarray}
  F(x,y)=\frac{\alpha(x,y)^2}{\beta(x,y)},
\end{eqnarray}
where $\alpha(x,y)=\sqrt{g_{ij}(x)y^iy^j}$ and $\beta(x,y)=b_i(x)y^i$.\\

In a natural way, the Riemannian metric $g$ induces an inner
product on any cotangent space $T^\ast_xM$ such that
$<dx^i(x),dx^j(x)>=g^{ij}(x)$. The induced inner product on
$T^\ast_xM$ induces a linear isomorphism between $T^\ast_xM$ and
$T_xM$ (for more details see \cite{DeHo}.). Then the 1-form $b$
corresponds to a vector field $\tilde{X}$ on $M$ such that
\begin{eqnarray}
  g(y,\tilde{X}(x))=\beta(x,y).
\end{eqnarray}

Therefore we can write the Kropina metric
$F=\frac{\alpha^2}{\beta}$ as follows:
\begin{eqnarray}
  F(x,y)=\frac{\alpha(x,y)^2}{g(\tilde{X}(x),y)}.
\end{eqnarray}
Flag curvature, which is a generalization of the concept of
sectional curvature in Riemannian geometry, is one of the
fundamental quantities which associates with a Finsler space. Flag
curvature is computed by the following formula:
\begin{eqnarray}\label{flag}
  K(P,Y)=\frac{g_Y(R(U,Y)Y,U)}{g_Y(Y,Y).g_Y(U,U)-g_Y^2(Y,U)},
\end{eqnarray}
where $g_Y(U,V)=\frac{1}{2}\frac{\partial^2}{\partial s\partial
t}(F^2(Y+sU+tV))|_{s=t=0}$, $P=span\{U,Y\}$,
$R(U,Y)Y=\nabla_U\nabla_YY-\nabla_Y\nabla_UY-\nabla_{[U,Y]}Y$ and
$\nabla$ is the Chern connection induced by $F$ (see \cite{BaChSh}
and \cite{Sh}.).\\
In general, the computation of the flag curvature of Finsler
metrics is very difficult, therefore it is important to find an
explicit and applicable formula for the flag curvature.  In
\cite{EsSa}, we have studied the flag curvature of invariant
Randers metrics on naturally reductive homogeneous spaces and in
\cite{Sa1} we generalized this study on a general homogeneous
space. Also in \cite{Sa2} we considered $(\alpha,\beta)-$metrics
of the form $\frac{(\alpha+\beta)^2}{\alpha}$ and gave the flag
curvature of these metrics. In this paper we study the flag
curvature of invariant Kropina metrics on homogeneous spaces.

\section{\textbf{Flag curvature of invariant Kropina metrics on homogeneous spaces}}
Let $G$ be a compact Lie group, $H$ a closed subgroup, and $g_0$ a
bi-invariant Riemannian metric on $G$. Assume that $\frak{g}$ and
$\frak{h}$ are the Lie algebras of $G$ and $H$ respectively. The
tangent space of the homogeneous space $G/H$ is given by the
orthogonal complement $\frak{m}$ of $\frak{h}$ in $\frak{g}$ with
respect to $g_0$. Each invariant metric $g$ on $G/H$ is determined
by its restriction to $\frak{m}$. The arising $Ad_H$-invariant
inner product from $g$ on $\frak{m}$ can extend to an
$Ad_H$-invariant inner product on $\frak{g}$ by taking $g_0$ for
the components in $\frak{h}$. In this way the invariant metric $g$
on $G/H$ determines a unique left invariant metric on $G$ that we
also denote by $g$. The values of $g_0$ and $g$ at the identity
are inner products on $\frak{g}$. We denote them by $<.,.>_0$
and $<.,.>$. The inner product $<.,.>$ determines a positive
definite endomorphism $\phi$ of $\frak{g}$ such that $<X,Y>=<\phi
X,Y>_0$ for all $X, Y\in\frak{g}$.\\
T. P\"uttmann has shown that the curvature tensor of the invariant
metric $<.,.>$ on the compact homogeneous space $G/H$ is given by
\begin{eqnarray}\label{puttmans formula}
  <R(X,Y)Z,W> &=& -\{\frac{1}{2}(<B_-(X,Y),[Z,W]>_0+<[X,Y],B_-(Z,W)>_0) \nonumber \\
    &+& \frac{1}{4}(<[X,W],[Y,Z]_{\frak{m}}>-<[X,Z],[Y,W]_{\frak{m}}> \\
    &-& 2<[X,Y],[Z,W]_{\frak{m}}>)+(<B_+(X,W),\phi^{-1}B_+(Y,Z)>_0 \nonumber\\
    &-&<B_+(X,Z),\phi^{-1}B_+(Y,W)>_0)\}\nonumber,
\end{eqnarray}
where $B_+$ and $B_-$ are defined by
\begin{eqnarray*}
  B_+(X,Y) &=& \frac{1}{2}([X,\phi Y]+[Y,\phi X]), \\
  B_-(X,Y) &=& \frac{1}{2}([\phi X,Y]+[X,\phi Y]),
\end{eqnarray*}
and $[.,.]_{\frak{m}}$ is the projection of $[.,.]$ to
$\frak{m}$.(see \cite{Pu}.).

\textbf{Notice.} We added a minus to the P\"uttmann's formula
because our definition of the curvature tensor $R$ is different
from the P\"uttmann's definition in a minus sign.

\begin{theorem}\label{flagcurvature}
Let $G, H, \frak{g}, \frak{h}, g, g_0$ and $\phi$ be as above.
Assume that $\tilde{X}$ is an invariant vector field on $G/H$ and
$X:=\tilde{X}_H$. Suppose that $F=\frac{\alpha^2}{\beta}$ is the
Kropina metric arising from $g$ and $\tilde{X}$ such that its
Chern connection coincides to the Levi-Civita connection of $g$.
Suppose that $(P,Y)$ is a flag in $T_H(G/H)$ such that $\{Y,U\}$
is an orthonormal basis of $P$ with respect to $<.,.>$. Then the
flag curvature of the flag $(P,Y)$ in $T_H(G/H)$ is given by
\begin{equation}\label{main-flag-cur-formula}
    K(P,Y)=\frac{3<U,X><R(U,Y)Y,X>+2<Y,X><R(U,Y)Y,U>}{2(\frac{<U,X>}{<Y,X>})^2+2},
\end{equation}
where
\begin{eqnarray}
   <R(U,Y)Y,X>&=&-\frac{1}{4}(<[\phi U,Y]+[U,\phi Y],[Y,X]>_0+<[U,Y],[\phi Y,X]+[Y,\phi X]>_0)\nonumber\\
    &&-\frac{3}{4}<[Y,U],[Y,X]_\frak{m}>-\frac{1}{2}<[U,\phi X]+[X,\phi U],\phi^{-1}([Y,\phi Y])>_0\\
    &&+\frac{1}{4}<[U,\phi Y]+[Y,\phi U],\phi^{-1}([Y,\phi X]+[X,\phi
    Y])>_0,\nonumber
   \end{eqnarray}
and
\begin{eqnarray}
  <R(U,Y)Y,U>&=&-\frac{1}{2}<[\phi U,Y]+[U,\phi Y],[Y,U]>_0\nonumber \\
  && \ \ \ -\frac{3}{4}<[Y,U],[Y,U]_{\frak{m}}>-<[U,\phi U],\phi^{-1}([Y,\phi Y])>_0 \\
  && \ \ \ +\frac{1}{4}<[U,\phi Y]+[Y,\phi U],\phi^{-1}([Y,\phi U]+[U, \phi Y])>_0.\nonumber
\end{eqnarray}
\end{theorem}

\begin{proof}
The Chern connection of $F$ coincides on the Levi-Civita
connection of $g$. Therefore the Finsler metric $F$ and the
Riemannian metric $g$ have the same curvature tensor. We
denote it by $R$.\\
By using the definition of $g_Y(U,V)$ and some computations for
$F$ we have:
\begin{eqnarray}\label{g_Y}
  g_Y(U,V)&=& \frac{1}{g^4(Y,X)}\{(2g(Y,U)g(Y,X)-g(U,X)g(Y,Y))(2g(Y,V)g(Y,X)-g(V,X)g(Y,Y))\nonumber \\
  &&+g(Y,Y)(g(Y,X)(2g(U,V)g(Y,X)+2g(Y,V)g(U,X)-2g(V,X)g(Y,U))\\
  &&-2g(U,X)(2g(Y,V)g(Y,X)-g(V,X)g(Y,Y)))\}\nonumber
\end{eqnarray}
By attention to this consideration that $\{Y,U\}$ is an
orthonormal basis for $P$ with respect to $g$ and (\ref{g_Y}) we
have
\begin{eqnarray}\label{eq1}
  g_Y(R(U,Y)Y,U)&=&\frac{1}{<Y,X>^4}\{\nonumber\\
  &&<U,X>(3<R(U,Y)Y,X>-2<Y,R(U,Y)Y><Y,X>)\\
  &&+2<Y,X>(<R(U,Y)Y,U><Y,X>-<U,X><Y,R(U,Y)Y>)\},\nonumber
\end{eqnarray}
and
\begin{eqnarray}\label{eq2}
  g_Y(Y,Y).g_Y(U,U)-g^2_Y(U,Y)&=&
  \frac{2<U,X>^2}{<Y,X>^6}+\frac{2}{<Y,X>^4}.
\end{eqnarray}
Now by using P\"uttmann's formula [6, eq. (\ref{puttmans formula})] we have:
\begin{eqnarray}\label{eq3}
<X,R(U,Y)Y>&=&-\frac{1}{4}(<[\phi U,Y]+[U,\phi Y],[Y,X]>_0+<[U,Y],[\phi Y,X]+[Y,\phi X]>_0)\nonumber\\
    &&-\frac{3}{4}<[Y,U],[Y,X]_\frak{m}>-\frac{1}{2}<[U,\phi X]+[X,\phi U],\phi^{-1}([Y,\phi Y])>_0\\
    &&+\frac{1}{4}<[U,\phi Y]+[Y,\phi U],\phi^{-1}([Y,\phi X]+[X,\phi
    Y])>_0,\nonumber
\end{eqnarray}

\begin{eqnarray}\label{eq4}
  <R(U,Y)Y,Y>=0,
\end{eqnarray}
and
\begin{eqnarray}\label{eq5}
  <R(U,Y)Y,U>&=&-\frac{1}{2}<[\phi U,Y]+[U,\phi Y],[Y,U]>_0\nonumber \\
  && \ \ \ -\frac{3}{4}<[Y,U],[Y,U]_{\frak{m}}>-<[U,\phi U],\phi^{-1}([Y,\phi Y])>_0 \\
  && \ \ \ +\frac{1}{4}<[U,\phi Y]+[Y,\phi U],\phi^{-1}([Y,\phi U]+[U, \phi Y])>_0.\nonumber
\end{eqnarray}
Substituting the equations (\ref{eq1}), (\ref{eq2}), (\ref{eq3}),
(\ref{eq4}) and (\ref{eq5}) in the equation (\ref{flag}) completes the
proof.

\end{proof}

Now we continue our study with a special type of Riemannian
homogeneous spaces which has been named naturally reductive. We
remind that  a homogeneous space $M=G/H$ with a $G-$invariant
indefinite Riemannian metric $g$ is said to be naturally reductive
if it admits an $ad(H)$-invariant decomposition
$\frak{g}=\frak{h}+\frak{m}$ satisfying the condition
\begin{eqnarray}
 B(X,[Z,Y]_{\frak{m}})+B([Z,X]_{\frak{m}},Y)=0 \hspace{1.5cm}\mbox{for} \ \ \ X, Y, Z \in
 \frak{m},
\end{eqnarray}
where $B$ is the bilinear form on $\frak{m}$ induced by $\frak{g}$
and $[,]_{\frak{m}}$ is the projection to $\frak{m}$ with respect
to the decomposition $\frak{g}=\frak{h}+\frak{m}$ (For more details see \cite{KoNo}.).\\
In this case the above formula for the flag curvature reduces to a
simpler equation.

\begin{theorem}
In the previous theorem let $G/H$ be a naturally reductive
homogeneous space. Then the flag curvature of the flag $(P,Y)$ in
$T_H(G/H)$ is given by \ref{main-flag-cur-formula} where,
\begin{eqnarray}
    R(U,Y)Y&=&\frac{1}{4}[Y,[U,Y]_{\frak{m}}]_{\frak{m}}+[Y,[U,Y]_{\frak{h}}]
\end{eqnarray}
\end{theorem}

\begin{proof}
By using Proposition 3.4 in \cite{KoNo} (page 202) the claim clearly follows.
\end{proof}

If the invariant Kropina metric is defined by a bi-invariant
Riemannian metric on a Lie group then there is a simpler formula
for the flag curvature, we give this formula in the following
theorem.
\begin{theorem}
Let $G$ be a Lie group, $g$ be a bi-invariant Riemannian metric on
$G$, and $\tilde{X}$ be a left invariant vector field on $G$.
Suppose that $F=\frac{\alpha^2}{\beta}$ is the Kropina metric
defined by $g$ and $\tilde{X}$ on $G$ such that the Chern
connection of $F$ coincides on the Levi-Civita connection of $g$.
Then for the flag curvature of the flag $P=span\{Y,U\}$, where
$\{Y,U\}$ is an orthonormal basis for $P$ with respect to $g$, we
have:
\begin{eqnarray}\label{flagbi-invariant}
  K(P,Y)=\frac{-3<U,X><[[U,Y],Y],X>-2<Y,X><[[U,Y],Y],U>}{8(\frac{<U,X>}{<Y,X>})^2+8},
\end{eqnarray}
\end{theorem}

\begin{proof}
$g$ is bi-invariant. Therefore we have
$R(U,Y)Y=-\frac{1}{4}[[U,Y],Y]$. Now by using Theorem
\ref{main-flag-cur-formula} the proof is completed.
\end{proof}

\large{\textbf{Acknowledgment}}\\
This work was supported by the research grant from Shahrood
University of Technology.


\bibliographystyle{amsplain}

\end{document}